\documentclass[12pt]{amsart}

\usepackage{lipsum}
\usepackage{amsfonts}
\usepackage{amsmath,amssymb}
\usepackage{graphicx}
\usepackage{epstopdf}
\usepackage{algorithmic}
\usepackage{color}
\ifpdf
  \DeclareGraphicsExtensions{.eps,.pdf,.png,.jpg}
\else
  \DeclareGraphicsExtensions{.eps}
\fi

\usepackage{enumitem}
\setlist[enumerate]{leftmargin=.5in}
\setlist[itemize]{leftmargin=.5in}

\theoremstyle{plain}
\newtheorem{theorem}{Theorem}[section]
\newtheorem{lemma}[theorem]{Lemma}

\theoremstyle{definition}
\newtheorem{definition}[theorem]{Definition}

\theoremstyle{remark}
\newtheorem{remark}[theorem]{Remark}

\newcommand{\R}{\mathbb{R}}

\newcommand{\M}{\mathcal{M}}
\newcommand{\PA}{\mathcal{P}}
\newcommand{\SA}{\mathcal{S}}
\newcommand{\X}{\mathcal{X}}
\newcommand{\0}{\textbf{0}}

\title[Exploding and vanishing gradients]{Exploding and vanishing gradients in  deep neural networks: the effect of residual connections}

\author{Vivek S Borkar}

\address{Department of Electrical Engineering, Indian Institute of Technology Bombay, 
Mumbai 400076, India.}
\email{borkar.vs@gmail.com}

\usepackage{amsopn}

\keywords{deep neural networks; exploding and vanishing gradients; multiplicative ergodic theorem; Liapunov exponents; ResNet}

\subjclass{68T07, 37H15}

\thanks{Work of VSB was supported by a grant from Google Research India and a National Science Chair from the Government of India.}

\begin{document}

\maketitle

\begin{abstract}
The well known phenomenon of exploding and vanishing gradients in deep neural networks is analyzed using multiplicative ergodic theory. The effect of adding a residual connection is explained in this context. Specifically, a characterization of Liapunov exponents due to Furstenberg and Kifer \cite{FK} is exploited in order to make a precise statement about the Liapunov spectrum and the effect of residual connections on it.
\end{abstract}




\section{Introduction}

One of the phenomena that has been observed in the training of Deep Neural Networks (DNNs) is that of exploding and decaying gradients. This refers to the empirically observed fact that on application of stochastic gradient descent (SGD) for their training, the gradients in certain directions may asymptotically tend to vanish or blow up in norm. There is a considerable literature for explaining this phenomenon and also for schemes to ameliorate this effect, see \cite{Allen-Zhu, Avelin, Chen, Engleken, Hanin1, Hanin2, Storm} for a small sampler. Our objective here is to present a different take on this problem using multiplicative ergodic theory and Liapunov exponents. In particular, we use an exact characterization of Liapunov exponents due to Furstenberg and Kifer \cite{FK} in order to precisely pin down the effect of residual connections. \\

To put our work in context, we briefly recall the approach taken in some of the prior works to explain this phenomenon and highlight the basic differences vis-a-vis the present effort. We discuss only a few representative works that are closest to ours in flavour, particularly in the common feature  of using multiplicative ergodic theory. \\

The closest to us in spirit is \cite{Storm}. Indeed this work applies multiplicative ergodic theory to flag the role of Liapunov exponents in shaping the asymptotic landscape of the input-output map of a DNN. What is new in our work is  the use of the exact characterization of the Liapunov exponents due to Fudenberg and Kifer \cite{FK} that allows us to go significantly further in order to pin down the Liapunov spectrum in greater detail and characterize the effect of residual connections on it in a precise mathematical fashion using projective geometry. \\

A similar comparison applies to \cite{Chen} which, under additional structural assumptions on the DNN, argues that residual connections will `tame' the Liapunov exponents in a certain sense, see \textit{ibid.} for details. An early work in this direction is \cite{Avelin}, which again applies multiplicative ergodic theory to DNNs in the manner that we do, but does not go for precise characterization of the Liapunov spectrum.\\

Other related works include \cite{Hanin1}, \cite{Hanin2}. In \cite{Hanin1}, random fully connected ReLU networks are considered and moments of norms of  Jacobian matrices of their input-output maps are estimated. In \cite{Hanin2}, again a random DNN is analyzed in the limit as its depth and the space dimension simultaneously go to infinity and asymptotic gaussianity of the norms  of the associated Jacobian matrix in a suitable scaling limit is established. See also \cite{Zhang} for some early work in this direction.\\ 

In an alternative approach, \cite{Allen-Zhu} frames DNN learning as hierarchical learning.\\

To summarize, the key distinction between the present work and its predecessors is the use of  the exact characterization of  Liapunov spectrum in multiplicative ergodic theory due to \cite{FK}, in order to make more precise statements about the effect of residual connections in DNNs.\\

This article is organized as follows. The next section, Section \ref{notation}, sets up the notation and states the problem. Section \ref{main}  recalls the key mathematical results from multiplicative ergodic theory that we plan to use, notably the aforementioned result due to Furstenberg and Kifer \cite{FK}.  Section \ref{resnet} explores the effect of residual connections on the DNNs, motivated by ResNet, using this result of \cite{FK}. \\

\section{Notation and Preliminaries}\label{notation}

We begin this section by describing up front the mapping from DNN to a discrete dynamical system which lies at the core of this work and some of the related works discussed above. This will be followed by setting up of the core notation that will be operative throughout. \\

\subsection{DNN as a dynamical system}

Fix $d \geq 1$. As a deep neural network, we consider a concatenation of $N \gg 1$ blocks of feedforward neural networks, with inputs and outputs from $\R^d$ to itself. This will be our DNN. (DNNs in general allow different dimensions for each layer. See Remark \ref{DIM} below.) \\

The component networks are assumed to belong to a parametrized family $f_\beta, \beta \in \R^s$ for some $s \geq 1$ and are assumed to be jointly continuously differentiable in $\beta$ and the input. This includes, e.g., sigmoidal networks, but not ReLU. We denote by $\beta^N$ the vector 
$$[\beta_N, \beta_{N-1}, \cdots , \beta_1] \in \R^{sN}$$
and by $f^N_{\beta^N} : \R^d \to \R^d$ the composition thereof, i.e.,
\begin{equation}
f^N_{\beta^N} := f_{\beta_N}\circ f_{\beta_{N-1}} \circ \cdots \circ f_{\beta_1} \circ f_{\beta_0} : \R^d \to \R^d. \label{DNNmap}
\end{equation}
Here `$\circ$' denotes composition of functions, i.e.\ $f\circ g(x) := f(g(x))$. \\

There is a standard way of mapping this composition into iterates of an equivalent discrete time dynamical system. Specifically, let $x_0 \in \R^d$ and recursively define 
\begin{equation}
x_{n+1} = f_{\beta_n}(x_n), \ n \geq 0. \label{dynamics0}
\end{equation}
Then inductively, one has
$$f^N_{\beta^N}(x_0) = x_N, N \geq 1$$
and vice versa. Thus the two descriptions are completely equivalent.\\

Let  $D(\beta)(x) :=$  the $d\times d$ Jacobian matrix of $f_\beta$ evaluated at $x$, parametrized by $\beta$. Likewise, let $D^N(\beta^N)(x^N):=$  the $d\times d$ Jacobian matrix of  $f^N$, parametrized by $\beta^N$ and evaluated at $x^N := [x_N, \cdots , x_0]$. Then by the chain rule of differentiation,
\begin{equation}
D^N(\beta^N)(x^N) = \prod_{m=0}^{N}D(\beta_{N-m})(x_{N-m}). \label{chainrule}
\end{equation}
Let $\eta_0\in \R^d$. We shall be interested in the asymptotic behaviour of 
\begin{eqnarray}
\eta^N &:=& \left(\prod_{m=0}^{N}D(\beta_{N-m})(x_{N-m})\right)\eta_0. \label{mult} \\
&=& D(\beta_N)(x_N)\eta^{N-1}. \label{dynamics1}
\end{eqnarray}
Equation \eqref{dynamics1} represents a time-inhomogeneous linear dynamical system with $\beta_n$'s serving as parameters. The second equality is the counterpart of the equivalence of  \eqref{DNNmap} and \eqref{dynamics0} mentioned above, but for the \textit{linearized dynamics} given by \eqref{dynamics1} as opposed to the original dynamics \eqref{dynamics0}.\\

Summarizing, this dynamical view of DNNs is tantamount to:\\

\begin{enumerate}
\item viewing input as the initial condition of a nonlinear dynamics,\\

\item the layer index $m \geq 0$ as a discrete time index,\\

\item the input-output map $f_{\beta_i}$ of the $i$th layer of the feedforward network as a time-dependent map that dictates the evolution of the dynamics at time $i$ from the state at time $i$ to the state at time $i+1$,\\

\item $\beta_i$'s are parameters of the next step map at time $i$, $i \geq 0$. 
 
 \end{enumerate}
 
 \medskip
 
 This equivalence allows us to apply techniques from the theory of nonlinear dynamics to DNNs as the depth of the DNN tends to infinity.\\
 
 \subsection{Notation}

 We shall use the following additional notation.\\
 
 \begin{enumerate}

\item We denote by $\M$ the set of $d\times d$ matrices and by $\M^+ \subset \M$ the subset thereof consisting of nonsingular $d\times d$ matrices. \\

\item We denote by $\0$ the zero vector in the appropriate dimension, depending on the context.\\

\item We denote by $\PA(\X)$ the Polish\footnote{i.e., a separable Hausdorff space which has a compatible complete metric} space of probability measures on the Polish space $\X$ with Prokhorov topology (also known as the topology of weak convergence). This topology is defined as the coarsest topology that renders continuous the maps $\zeta \in \PA(\X) \mapsto \int fd\zeta\in \R$ for $f \in C_b(\X) :=$ the space of bounded continuous functions on $\X$. (See, e.g., \cite{Bill} for a detailed exposition of this topology.)\\

\item We denote by $S^d$ the unit sphere in $\R^d$ and denote by $\Phi : \R^d\backslash\{\0\} \to S^d$ the map $x \mapsto \frac{x}{\|x\|}$ when $x \neq \0$.\\

\item We denote by $P^{d-1}$ the projective space, i.e., the space of  equivalence classes of nonzero vectors in $\R^d$ under the equivalence relation $x \equiv \lambda x$ for $\lambda \in \R\backslash\{\0\}$. If $u \in P^{d-1}$, we denote by $\hat{u}$ a generic element of $\R^d$ that gets mapped to $u$ under this equivalence relation. \\

\item We denote by $\Psi : \R^d\backslash\{\0\}  \to P^{d-1}$ the map that maps $x \in \R^d$ to the corresponding element of $P^{d-1}$. $P^{d-1}$ is endowed with the natural metric 

$$d(u,u') \ := \ \min\left(\left\|\frac{\hat{u}}{\|\hat{u}\|}-\frac{\hat{u}'}{\|\hat{u}'\|}\right\| \ , \ \left\|\frac{\hat{u}}{\|\hat{u}\|}+\frac{\hat{u}'}{\|\hat{u}'}\right\|\right),$$

\medskip

for $\hat{u}\in \Psi^{-1}(u)$ and $\hat{u}' \in \Psi^{-1}(u')$.\\

\item $\mu$ is a prescribed probability measure on $\M$.

\end{enumerate}

\medskip

We next recall the Furstenberg-Kifer theorem of multiplicative ergodic theory, which will be the basis of our analysis of DNNs.

\section{Furstenberg-Kifer theorem}\label{main}

Multiplicative ergodic theorems go back to \cite{Osel}. The basic result and the associated theory of Liapunov exponents has been extended in many directions, see \cite{Arnold} for a comprehensive treatment. We use here a variant due to Furstenberg and Kifer \cite{FK} which has the additional feature of giving a clean characterization of the Liapunov spectrum and the resulting direct sum decomposition of the state space, in terms of an associated Markov chain on the projective space. This needs the following additional assumption.\\

\noindent \textbf{Assumption 1:}\label{Ass} $$\int\left(\log^+\|A\| + \log^+\|A^{-1}\|\right)\mu(dA) < \infty.$$

\medskip

%

\begin{definition}\label{MKernel} Given a $\mu \in \PA(\M)$ and a $\nu \in \PA(P^{d-1})$, we define $\mu*\nu \in \PA(P^{d-1})$ by
\begin{equation}
\int \varphi d(\mu*\nu) := \int \varphi(Ax)\mu(dA)\nu(dx) \ \ \ \ \forall \ \varphi \in C_b(\M). \label{star}
\end{equation}
\end{definition}

\medskip

\begin{lemma}\label{extreme} The set $\SA := \{\nu \in \PA(P^{d-1}) : \mu*\nu = \nu\} \subset \PA(P^{d-1})$ is a nonempty  and compact simplex whose extreme points are mutually singular. \end{lemma}

\begin{proof} This is immediate from the discussion in \cite{FK}, pp.\ 17-18, where the correspondence of $\SA$ with the stationary distributions of a Markov chain is established. Specifically, define the transition kernel 
$$\kappa : P^{d-1} \to \PA(P^{d-1})$$ 
by
$$\int f(y)\kappa(dy|x) := \int_{P^{d-1}}f\left(\frac{Ax}{\|Ax\|}\right)\mu(dA) \ \ \forall \ f \in C\left(P^{d-1}\right).$$
(Equivalently, $$\kappa(B|x) := \int_{P^{d-1}}I\left\{\left(\frac{Ax}{\|Ax\|}\right)\in B\right\}\mu(dA)$$
for all Borel $B \subset P^{d-1}$.) \\

Thus $\nu$ is simply an invariant measure of this transition kernel by Definition \ref{MKernel}. Since $P^{d-1}$ is compact and the map $x \to \kappa(dy|x)$ is  seen to be continuous in $x$, it follows from standard Markov process theory \cite{Benaim}, \cite{MeynTweedie} that the set of such probability measures $\nu$ forms a nonempty compact simplex whose extreme points are mutually singular. \end{proof}

\medskip

We call $\SA$ the set of $\mu$-stationary measures in $\PA(P^{d-1})$. We shall also need the following definitions.\\

\begin{enumerate}

\item For a subspace $Z$ of $\R^d$, denote by $\overline{Z}$ the set of corresponding elements of $P^{d-1}$.\\

\item For $\mu \in \PA(\M)$, say that a subspace $Z$ of $\R^d$ is  \textit{$\mu$-invariant} if it is invariant under $\mu$-a.s.\ $A \in \M$.
\end{enumerate}

\begin{remark}\label{DIM} A priori, our assumption that all $f_{\beta_i}$'s map $\R^d$ to itself is not restrictive. We can take the $d < \infty$ to be an upper bound (assumed to exist) on the dimensionality of the input or output space of the individual networks and set the appropriate connection weights to zero if the actual dimension is lower. However,  Assumption 1 above is  restrictive. It is required in order to be able to use the theory of \cite{FK}.  Our aim is to demonstrate a plausible mechanism for explaining the observed benefits of residual connections in a quantitative fashion, albeit for a stylized model. Since matrices satisfying Assumption 1 are dense in $\M$, this exercise is not entirely unreasonable. That said, it will need a lot more sophisticated mathematics (in particular, a suitable extension of Theorem \ref{FuKe} in order to push these results to full generality, which is a task for the future. \end{remark}

\medskip

With this notation and caveats, we next state the key results from \cite{FK} that are relevant for our purposes. These have been recast in our notation.\\

Under stated hypotheses, \cite{FK} proves the following (See Theorems 3.9 and 3.10 of \textit{ibid.}).

\begin{theorem}\label{FuKe} $(i)$ There exist an integer $1 \leq r \leq d$, a sequences of subspaces of $\R^d$
$$\{0\} \subset L_r \subset L_{r-1} \subset \cdots \subset L_2 \subset L_1 \subset L_0 = \R^d$$
and a sequence of real numbers
$$\gamma(\mu) := \gamma^0(\mu) > \gamma^1(\mu) > \gamma^2(\mu) > \cdots > \gamma^r(\mu)$$
such that, if $v\in L_i\backslash L_{i+1}$, then for $\{X_n\}$ i.i.d.\ with law $\mu$,
$$\lim_{N\uparrow\infty}\frac{1}{N}\log\left\|X_NX_{N-1}\cdots X_1v\right\| = \gamma^i(\mu).$$

$(ii)$ These $\{\gamma^i(\mu), 0 \leq i \leq r\}$ are precisely the discrete values taken by the quantity
\begin{equation}
\alpha(\mu,\nu) := \int_\M\int_{P^{d-1}}\log\left(\|Au\|\right)\mu(dA)\nu(du) \label{alpha}
\end{equation}
as $\nu$ varies over all $\mu$-stationary measures. Also, $L_i :=$ the unique maximal $\mu$-invariant subspace of the set of all $\mu$-invariant subspaces satisfying $\nu(\overline{L}_i) = 0$ for all $\nu$ with $\alpha(\mu,\nu) > \gamma^i(\mu)$.
\end{theorem}

The $\{\gamma^i_\mu\}$ are called the Liapunov spectrum associated with the i.i.d.\ matrices $\{X_n\}$. The significant part of Theorem \ref{FuKe} for our purpose is part $(ii)$, which characterizes the Liapunov spectrum and the associated direct sum decomposition of the state space in terms of a Markov chain.\\

\section{Effect of residual connections}\label{resnet}

In order to map our problem to this framework, we consider $\{\beta_n\}$ i.i.d.\ with law (say) $\Gamma \in \PA(\R^s)$. We make the following additional assumption:\\

\noindent \textbf{Assumption 2:} The limit $x_\infty := \lim_{n\uparrow\infty}x_n$ exists a.s.\ in \eqref{dynamics0}.\\

We justify this as follows. In classifier DNNs, there are finitely many classes encoded as euclidean vectors $C = \{c_1, \cdots , c_k\} \subset \R^d$ (say) at the output and for $\mu$-a.s. $x_0$, the above limit $x_\infty \in C$ is well defined. Fix $x_0 = v$ in the probability $1$ set where this holds. Then Assumption 2 holds.\\

As we are considering the asymptotic regime, in what follows, we consider $x_n$ replaced by $x_\infty$. Furthermore, since $x_\infty$ is measurable with respect to the tail $\sigma$-field $\cap_{n \geq 1}\sigma(\beta_i, i \geq n)$ which is trivial by the Kolmogorov $0-1$ law, $x_\infty$ is a.s. a constant and we may take it to be a deterministic constant. Thus $D(\beta_n)(x_\infty), n \geq 0,$ are i.i.d.\ with law (say) $\upsilon$.\\

The ResNet architecture can be viewed as adding a forward connection, the so called `residual connection', that replaces every forward block $x \mapsto f_{\beta_i}(x)$ of the DNN by the map $x \mapsto x + f_{\beta_i}(x)$. We now explore how this affects the overall input-output map of the DNN in the limit as $N \uparrow \infty$, in view of the foregoing.\\

We do this by comparing the above maps for $x = \hat{u}$ as defined earlier with $\|\hat{u}\| = c > 0$ (say).  Since $\mu$ and therefore the set $\SA$ of possible values of $\nu$ is fixed, we focus on the quantity 
$$\xi(u, A) := \log(\|Au\|) = \log\left(\frac{\|A\hat{u}\|}{\|\hat{u}\|}\right),$$
and explore how it changes when $A$ is replaced by $I + A$.  \\

%
%
%

Now we are ready to prove our main result. \\

\begin{theorem}\label{final}  
The Liapunov spectrum under a residual connection is a smaller perturbation of the spectrum of the identity matrix (i.e., the vector of all $1$'s) than without a residual connection.
\end{theorem}

\begin{proof} It is easy to check that $\Phi(\hat{u}), \Psi(\hat{u})$ are independent of the $c = \|\hat{u}\|$ above from their very definition. Hence we can and do take $c = 1$ without any loss of generality. Now consider  
$$\hat{u}_0 := \hat{u}, \ \hat{u}_1 := A\hat{u}_0, \ \hat{u}_2 := (I + A)\hat{u}_0 = \hat{u}_0 + A\hat{u}_0.$$ 
Consider the two dimensional parallelopiped $B$ formed by $\0, \hat{u}_0, \hat{u}_1$ and $\hat{u}_2$. Let $\hat{u}_3:=$ the intersection of its diagonals. Then $\hat{u}_3$ is in the relative interior of $B$. Let $\hat{u}_i' := \Phi(\hat{u}_i), i = 0,1,3$. Then $u_i' \in B\cap S^d$ for $ i = 0,1,3$. It is then easy to see that $\hat{u}_3'$ lies in the relative interior of the arc $B\cap S^d$ joining $\hat{u}_0'$ and $\hat{u}_1'$. Therefore it is closer to $\hat{u}_0'$ than $\hat{u}_1'$. This property is  preserved under the map $\Psi$. That is, on mapping these vectors to the corresponding equivalence classes in $P^{d-1}$, $\Psi(\hat{u}_3')$ lies closer to $\Psi(\hat{u}_0)$ than $\Psi(\hat{u}_1)$, in the metric topology of $P^{d-1}$. 
Hence it follows that
\begin{eqnarray*}
\lefteqn{\int_{P^{d-1}}\int_\M \log\left(\|(I + A)u\|\right))\varphi(dA)\nu(du) \ -} \\
&& \  \int_{P^{d-1}}\int_\M \log\left(\|Au\|\right))\varphi(dA)\nu(du) \ < \ 0.
\end{eqnarray*}
The claim follows. \end{proof}

\medskip

 This result explains the stabilizing effect of residual connections on the Liapunov exponents of the DNN in a precise quantifiable sense.

\end{document}